\documentclass[12pt]{article}
\usepackage{fullpage}
\usepackage{mathrsfs,pstricks,ifpdf,tikz}
\usepackage{amsfonts}
\usepackage{enumerate}

\usepackage{makecell}
\newcommand{\tabincell}[2]{\begin{tabular}{@{}#1@{}}#2\end{tabular}}
\usepackage{latexsym}
\usepackage{amsmath,amssymb}
\usepackage{amsthm}
\usepackage{ifthen,calc}
\usepackage{color}
\usepackage{graphicx}
\usepackage{latexsym}
\usepackage{multirow}
\usepackage{diagbox}
\usepackage{float}
\usepackage[format=hang, margin=10pt]{caption}

\newtheorem{theorem}{Theorem}[section]

\newtheorem{lemma}[theorem]{Lemma}

\theoremstyle{definition}
\newtheorem{definition}[theorem]{Definition}

\newcommand{\ssum} {\sum_{i=1}^{\infty} }
\def\dsum{\displaystyle\sum}
\def\rs{\rm\scriptsize}

\newcommand{\mdef}[1]{\textit{\textbf{#1}}}  
\newcommand{\vsa}{\vskip-12pt}
\newcommand{\vsb}{\vskip-6pt}
\newcommand{\Z}{{\mathbb Z}}

\def\CL#1{\left\lceil#1\right\rceil}
\def\FL#1{\left\lfloor#1\right\rfloor}
\def\VEC#1#2#3{#1_{#2},\ldots,#1_{#3}}
\def\UM#1#2{\bigcup_{{#1}\in{#2}}}
\def\FR#1#2{\frac{#1}{#2}}
\def\NN{{\mathbb N}}
\def\C#1{\left\vert{#1}\right\vert}
\def\wG{{\widehat G}}
\def\wB{{\widehat B}}
\def\vB{{\tilde B}}

\renewcommand{\baselinestretch}{1.15}
\begin{document}

\title{On the bar visibility number of complete bipartite graphs}

\author{
Weiting Cao\thanks{University of Illinois, Urbana, USA: caoweiting@gmail.com.}\,
\qquad
Douglas B. West\thanks{Zhejiang Normal University, Jinhua, China, and
University of Illinois, Urbana, USA: dwest@math.uiuc.edu.
Supported by NNSF of China under Grant NSFC-11871439.}
\qquad
Yan Yang\thanks{Tianjin University, Tianjin, China: yanyang@tju.edu.cn.
Supported by NNSF of China under Grant NSFC-11401430.}\,
}

\date{\today}

\maketitle

\begin{abstract}
A {\it $t$-bar visibility representation} of a graph assigns each vertex up to
$t$ horizontal bars in the plane so that two vertices are adjacent if and only
if some bar for one vertex can see some bar for the other via an unobstructed
vertical channel of positive width.  The least $t$ such that $G$ has a $t$-bar
visibility representation is the {\it bar visibility number} of $G$, denoted by
$b(G)$.  For the complete bipartite graph $K_{m,n}$, the lower bound
$b(K_{m,n})\ge\lceil{\frac{mn+4}{2m+2n}}\rceil$ from Euler's Formula is
well known.  We prove that equality holds.\\



\end{abstract}

\section{Introduction}

In computational geometry, graphs are used to model visibility relations in the
plane.  For example, we may say that two vertices of a polygon ``see'' each
other if the segment joining them lies inside the polygon.  In the
{\it visibility graph} on the vertex set, vertices are adjacent if they see
each other.  More complicated notions of visibility have been defined for
families of rectangles and other geometric objects.  Dozens of papers
have been written concerning construction and recognition of visibility
graphs and applications to search problems and motion planning.  For a textbook
on algorithms for visibility problems, see Ghosh~\cite{Gho}.

We consider visibility among horizontal segments in the plane.  A graph $G$ is
a \textit{bar visibility graph} if each vertex can be assigned a horizontal
line segment in the plane (called a {\it bar}) so that vertices are adjacent if
and only if the corresponding bars can see each other along an unobstructed
vertical channel with positive width.  The assignment of bars is a
\textit{bar visibility representation} of $G$.  The condition on positive width
allows bars $[(a,y),(x,y)]$ and $[(x,z),(c,z)]$ to block visibility at $x$
without seeing each other.

Tomassia and Tollis~\cite{Tamassia86} and Wismath~\cite{Wismath85} found a
simple characterization of bar visibility graphs.  Hutchinson~\cite{Hut}
later gave another simple proof for the $2$-connected case.

\begin{theorem}[\cite{Tamassia86, Wismath85}]\label{4}
A graph $G$ has a bar visibility representation if and only if for some
planar embedding of $G$ all cut-vertices appear on the boundary of one face.
\end{theorem}

Theorem~\ref{4} is quite restrictive.  Nevertheless, assigning multiple bars to
vertices permits representations of all graphs and leads to a complexity
parameter measuring how many bars are needed per vertex, introduced by Chang,
Hutchinson, Jacobson, Lehel, and West \cite{Chang04}.

\begin{definition}[\cite{Chang04}]
A \emph{$t$-bar visibility representation} of a graph assigns to each vertex
at most $t$ horizontal bars in the plane so that vertices are adjacent if and
only if some bar assigned to one sees some bar assigned to the other via
an unobstructed vertical channel of positive width.
The \emph{bar visibility number} of a graph $G$, denoted by $b(G)$, is the
least integer $t$ such that $G$ has a $t$-bar visibility representation.
\end{definition}


Results in \cite{Chang04} include the determination of visibility number for
planar graphs (always at most $2$), plus $b(K_n)=\CL{n/6}$ for $n\ge7$, the
determination of $b(K_{m,n})$ within $1$, and $b(G)\le \CL{n/6}+2$
for every $n$-vertex graph $G$.  Results on the visibility numbers for
hypercubes~\cite{West17} and an analogue for directed graphs~\cite{ABHW} have
also been obtained. For complete bipartite graphs, the result was as follows.

\begin{theorem}[\cite{Chang04}]\label{1}
$r\leq b(K_{m,n})\leq r+1$, where $r=\CL{\frac{mn+4}{2m+2n}}$.
\end{theorem}

\noindent
To prove the lower bound, consider a $t$-bar representation, add edges to
encode visibilities that produce edges of $K_{m,n}$, and then shrink bars to
single points.  This produces a bipartite plane graph $H$ with at most $t(m+n)$
vertices and at least $mn$ edges.  Hence $mn\le 2t(m+n)-4$ by Euler's Formula,
so $b(K_{m,n})\ge r$.  Equality requires most faces in $H$ to have length $4$.

In this paper, we prove $b(K_{m,n})=r$.  Section $2$ contains a short proof
valid for $K_{n,n}$.  For this case, it suffices to decompose the graph into
$r$ bar visibility graphs, where a \textit{decomposition} of $G$ is a set of
edge-disjoint subgraphs whose union is $G$.  The subgraphs can then be
repesented with disjoint projections on the horizontal axis.  In Section $3$,
we present a different approach that solves the problem for all complete
bipartite graphs.

Our results are related to earlier work.  A {\it $t$-split} of a graph $G$ is a
graph $H$ in which each vertex is replaced by a set of at most $t$ independent
vertices in such a way that $u$ and $v$ are adjacent in $G$ if and only if some
vertex in the set representing $u$ is adjacent in $H$ to some vertex in the set
representing $v$.  The graph $G$ used to prove the lower bound for
Lemma~\ref{1} is an example of a $t$-split of $K_{m,n}$.  As defined by
Eppstein et al.~\cite{Epp}, the {\it planar split thickness}
(or simply {\it split thickness}) of a graph $G$, which we denote by
$\sigma(G)$, is the minimum $t$ such that $G$ has a $t$-split that is a
planar graph.  As explained above, always $\sigma(G)\le b(G)$.  If $G$ has
a $\sigma(G)$-split that is $2$-connected, then $\sigma(G)=b(G)$.

This connection was noted earlier in the thesis of the first author~\cite{Cao},
where planar split thickness was given the unfortunate name ``split number'',
creating confusion with another concept.  The {\it splitting number} of a
graph is the minimum number of successive splits of one vertex into two
(with each incident edge being inherited by one of the two new vertices)
needed to produce a planar graph.

The notion of $t$-split originated with Heawood~\cite{Hea}, who proved
that $K_{12}$ has a $2$-split.  Later, Ringel and Jackson~\cite{RJ}
proved in effect that $K_n$ has a $\CL{n/6}$-split.  A short proof of this
by Wessel~\cite{Wes} was used in~\cite{Chang04} to prove $b(K_n)=\CL{n/6}$.

The results in~\cite{Epp} that concern complete bipartite graphs determine
those that are $2$-splittable.  They are the same as those having bar
visibility number at most $2$.  Their lower bounds on $\sigma(K_{m,n})$ use
the same counting argument from Euler's Formula that yields the lower bounds
for $b(K_{m,n})$ (see~\cite{Chang04}).

In~\cite{Epp}, the authors close the paper by asking whether graphs embeddable
on the surface of genus $k$ are $(k+1)$-splittable, as an open question.
This follows from a recent result about the {\it thickness} $\theta(G)$ of a
graph $G$, defined to be the minimum number of planar graphs needed to
decompose $G$.  A decomposition into $k$ planar graphs is a $k$-split,
so $\sigma(G)\le \theta(G)$; this motivates the term ``split thickness''.
Xu and Zha~\cite{XZ} proved that $\theta(G)\le k+1$ when $G$ embeds on the
surface of genus $k$, thereby providing a positive answer to the question
in~\cite{Epp}.

\section{The bar visibility number of $K_{n,n}$}

As noted above, thickness provides an upper bound on the split thickness,
and the split thickness usually equals the bar visibility number.  Beineke,
Harary, and Moon~\cite{BHM64} determined $\theta(K_{m,n})$ for most $m$ and $n$.

\begin{lemma}[\cite{BHM64}]\label{3a}
$\theta(K_{n,n})=\CL{\frac{n+2}{4}}$.
\end{lemma}

When $\theta(K_{n,n})$ is the desired value for $b(K_{n,n})$, we aim to
decompose $K_{n,n}$ into that number of bar visibility graphs.  The difficult
case is when $b(K_{n,n})<\theta(K_{n,n})$.

\begin{theorem}\label{2}
$b(K_{n,n})=\CL{\frac{n+1}{4}}$, except for $b(K_{3,3})=2$.
\end{theorem}

\begin{proof}
It is immediate that $K_{1,1}$ and $K_{2,2}$ are bar visibility graphs.
Since $K_{3,3}$ is not planar, $b(K_{3,3})\ge 2$; equality holds because
$K_{3,3}$ decomposes into a $6$-cycle and a matching of size $3$, both of
which are bar visibility graphs.  Hence we may assume $n\ge 4$.

Let $r=\CL{(n+1)/4}$.  When $\theta(K_{n,n})=r$, we will decompose $K_{n,n}$
into $r$ bar visibility graphs.  This will leave the case where
$n\equiv 3\mod4$, in which case $r<\theta(K_{n,n})$ and $K_{n,n}$ cannot
decompose into $r$ bar visibility graphs.
Let $U$ and $V$ be the parts of $K_{n,n}$, with $U=\{u_1,\ldots, u_{n}\}$ and
$V=\{v_1,\ldots, v_{n}\}$.  Let $p=\FL{n/4}$.

For $n\equiv 0\mod4$, Chen and Yin \cite{CYi16} provided a decomposition of
$K_{n,n}$ into $p+1$ planar subgraphs $\{G_1,\ldots,G_{p+1}\}$.  Let
$[p]=\{1,\dots,p\}$.  For $1\le j\le p$, let
$U_1^j=\UM i{[p]-\{j\}}\{u_{4i-3},u_{4i-2}\}$, let
$U_2^j=\UM i{[p]-\{j\}}\{u_{4i-1},u_{4i}\}$, let
$V_1^j=\UM i{[p]-\{j\}}\{u_{4i-3},u_{4i-1}\}$, and let
$V_1^j=\UM i{[p]-\{j\}}\{u_{4i-2},u_{4i}\}$.
Figure 1 shows the subgraph $G_j$, for $1\le j\le p$.
Being a $2$-connected planar graph, it is a bar visibility graph.
The subgraph induced by the eight special vertices $\VEC u{4j-3}{4j}$ and
$\VEC v{4j-3}{4j}$ is $K_{4,4}$ minus the edges of the form $u_iv_i$.  The
remaining graph $G_{p+1}$ is the matching consisting of $u_iv_i$ for
$1\le i\le 4p$.  Again this is a bar visibility graph.

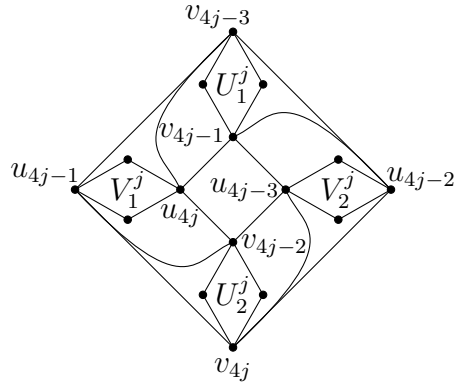
\begin{figure}[h!]
\begin{center}
\begin{tikzpicture}                                                         
[inner sep=0pt]
\filldraw [black] (0,2.1) circle (1.4pt)
                  (2.1,0) circle (1.4pt)
                  (2.1,4.2) circle (1.4pt)
                  (4.2,2.1) circle (1.4pt);
\filldraw [black]     (1.4,2.1) circle (1.4pt)
                      (2.8,2.1) circle (1.4pt)
                      (2.1,1.4) circle (1.4pt)
                      (2.1,2.8) circle (1.4pt);
\filldraw [black] (0.7,2.5) circle (1.4pt)
                  (0.7,1.7) circle (1.4pt);
\filldraw [black]      (3.5,2.5) circle (1.4pt)
                       (3.5,1.7) circle (1.4pt);
\filldraw [black] (2.5,0.7) circle (1.4pt)
                  (1.7,0.7) circle (1.4pt);
\filldraw [black]      (2.5,3.5) circle (1.4pt)
                       (1.7,3.5) circle (1.4pt);
\draw (0,2.1)--(2.1,0);\draw (2.1,0)--(4.2,2.1);\draw (2.1,4.2)--(0,2.1);
\draw (4.2,2.1)--(2.1,4.2);
\draw (1.4,2.1)--(2.1,1.4);\draw (2.1,1.4)--(2.8,2.1);\draw (2.8,2.1)--(2.1,2.8);\draw (1.4,2.1)--(2.1,2.8);
\draw (0,2.1)--(0.7,2.5);\draw (0,2.1)--(0.7,1.7);
\draw (1.4,2.1)--(0.7,2.5);\draw (1.4,2.1)--(0.7,1.7);
\draw (2.8,2.1)--(3.5,2.5);\draw (2.8,2.1)--(3.5,1.7) ;
\draw (4.2,2.1)--(3.5,2.5);\draw (4.2,2.1)--(3.5,1.7) ;
\draw (2.1,0)--(2.5,0.7);\draw (2.1,0)--(1.7,0.7);
\draw (2.1,1.4)--(2.5,0.7);\draw (2.1,1.4)--(1.7,0.7);
\draw (2.1,2.8)--(2.5,3.5);\draw (2.1,2.8)--(1.7,3.5);
\draw (2.1,4.2)--(2.5,3.5);\draw (2.1,4.2)--(1.7,3.5);
\draw (2.1,-0.3) node {$v_{4j}$}
      (1.9,4.4) node {$v_{4j-3}$}
      (-0.4,2.3) node {$u_{4j-1}$}
      (4.6,2.25) node {$u_{4j-2}$};
\draw    (2.65,1.35) node {$v_{4j-2}$}
          (1.55,2.83) node {$v_{4j-1}$}
           (1.4,1.75) node {$u_{4j}$}
             (2.25,2.1) node {$u_{4j-3}$};
\draw[-](2.1,4.2)..controls+(-1.6,-1.5)and+(-0.2,0.5)..(1.4,2.1);
\draw[-](0,2.1)..controls+(1.5,-1.6)and+(-0.5,-0.2)..(2.1,1.4);
\draw[-](2.1,0)..controls+(1.6,1.5)and+(0.2,-0.5)..(2.8,2.1);
\draw[-](4.2,2.1)..controls+(-1.5,1.6)and+(0.5,0.2)..(2.1,2.8);
\draw (.7,2.1) node {$V_1^j$};
\draw (3.5,2.1) node {$V_2^j$};
\draw (2.1,0.7) node {$U_2^j$};
\draw (2.1,3.5) node {$U_1^j$};
\end{tikzpicture}
%
\caption{The graph $G_j$ in a planar decomposition of $K_{4p,4p}$.}
\label{figure 1}
\end{center}
\end{figure}

For $n=4p+1$, we add two vertices $u_{4p+1}$ and $v_{4p+1}$,
with $u_{4p+1}$ adjacent to $V$ and $v_{4p+1}$ adjacent to $U$.
The edges incident to $u_{4p+1}$ and $v_{4p+1}$ can be added to the
graph $G_{p+1}$ of the previous case, as shown in Figure 2.
Again this graph is planar and $2$-connected, so again we have a
decomposition $\VEC{\tilde G}1{p+1}$ into $p+1$ bar visibility graphs.

\begin{figure}[htp]
\begin{center}
\vskip-0.8cm\begin{tikzpicture}
[xscale=1]
[inner sep=0pt]
\filldraw [black] (0,1) circle (1.4pt)
                   (0,-0.5) circle (1.4pt)
                  (1,1) circle (1.4pt)
                   (1,-0.5) circle (1.4pt)
                   (3,1) circle (1.4pt)
                   (3,-0.5) circle (1.4pt)
                  (4,1) circle (1.4pt)
                  (4,-0.5) circle (1.4pt);
\draw(4,1)-- (4,-0.5);\draw(3,1)-- (3,-0.5);
\draw(1,1)-- (1,-0.5);\draw(0,1)-- (0,-0.5);
\filldraw [black] (1.8,0.3) circle (0.8pt)
                               (2,0.3) circle (0.8pt)
                               (2.2,0.3) circle (0.8pt);
\draw (-0.3,0.9) node {$u_{1}$}
      (-0.3,-0.4) node {$v_{1}$}
      (0.7,0.9) node {$u_{2}$}
      (0.7,-0.4) node {$v_{2}$};
\draw (2.4,0.9) node {$u_{4p-1}$}
      (2.4,-0.4) node {$v_{4p-1}$}
      (4.4,0.9) node {$u_{4p}$}
      (4.35,-0.4) node {$v_{4p}$};
\filldraw [black] (2,2) circle (1.4pt)
                  (2,-1.5) circle (1.4pt);
\draw(2,2)-- (0,1);\draw(2,2)-- (1,1);\draw(2,2)-- (3,1);\draw(2,2)-- (4,1);
\draw(2,-1.5)-- (0,-0.5);\draw(2,-1.5)-- (1,-0.5);\draw(2,-1.5)-- (3,-0.5);\draw(2,-1.5)-- (4,-0.5);
\draw (2.5,2.1) node {$v_{4p+1}$}
      (2.5,-1.7) node {$u_{4p+1}$};
\draw[-](2,2)..controls+(-3.8,1)and+(-3.8,-1)..(2,-1.5);
\end{tikzpicture}
\vskip-0.8cm\caption{The subgraph $\widetilde{G}_{p+1}$ in the planar
decomposition of $K_{4p+1,4p+1}$}
\label{figure 2}
\end{center}
\end{figure}

For $n=4p+2$, we modify the decomposition given for $K_{4p,4p}$ to accommodate
the edges incident to $\{u_{4p+1},u_{4p+2},v_{4p+1},v_{4p+2}\}$.
First form $\widehat G_{p+1}$ by adding to the matching $G_{p+1}$ the edges
joining $u_{4p+1}$ to $\UM i{[p]}\{v_{4i-2},v_{4i}\}$,
joining $u_{4p+2}$ to $\UM i{[p]}\{v_{4i-3},v_{4i-1}\}$,
joining $v_{4p+1}$ to $\UM i{[p]}\{v_{4i-2},v_{4i-3}\}$, and
joining $v_{4p+2}$ to $\UM i{[p]}\{v_{4i},v_{4i-1}\}$,
plus the edges of the $4$-cycle $[u_{4p+1},v_{4p+1},u_{4p+2},v_{4p+2}]$,
as shown in Figure 3.
To include the remaining edges involving the four added vertices,
for $1\le j\le p$ obtain $\widehat G_j$ from $G_{j}$ by adding $u_{4p+i}$ to
$U_i^j$ and $v_{4p+i}$ to $V_i^j$, for $i\in\{1,2\}$.  Each of these four
vertices gains the two neighbors in $G_j$ that are shared by the vertices of
the set to which it was added.  Over the resulting $\VEC{\widehat G}1p$, it
gains precisely the neighbors in the other part that it does not have in
$\widehat G_{p+1}$.  We again have $r$ $2$-connected planar graphs decomposing
$K_{n,n}$.

\begin{figure}[htp]
\begin{center}
\begin{tikzpicture}
[xscale=1]
[inner sep=0pt]
\filldraw [black] (0,1) circle (1.4pt)
                   (0,-0.5) circle (1.4pt)
                  (1,1) circle (1.4pt)
                   (1,-0.5) circle (1.4pt)
                   (1.5,1) circle (1.4pt)
                   (1.5,-0.5) circle (1.4pt)
                   (2.5,1) circle (1.4pt)
                   (2.5,-0.5) circle (1.4pt)
                   (3,1) circle (1.4pt)
                   (3,-0.5) circle (1.4pt)
                  (4,1) circle (1.4pt)
                  (4,-0.5) circle (1.4pt)
                    (5.5,1) circle (1.4pt)
                   (5.5,-0.5) circle (1.4pt)
                   (6.5,1) circle (1.4pt)
                   (6.5,-0.5) circle (1.4pt)
                   (7,1) circle (1.4pt)
                   (7,-0.5) circle (1.4pt)
                  (8,1) circle (1.4pt)
                  (8,-0.5) circle (1.4pt)
                  (8.5,1) circle (1.4pt)
                   (8.5,-0.5) circle (1.4pt)
                   (9.5,1) circle (1.4pt)
                   (9.5,-0.5) circle (1.4pt);
\draw(4,1)-- (4,-0.5);\draw(3,1)-- (3,-0.5);\draw(1.5,1)-- (1.5,-0.5);\draw(2.5,1)-- (2.5,-0.5);\draw(5.5,1)-- (5.5,-0.5);
\draw(7,1)-- (7,-0.5);\draw(8,1)-- (8,-0.5);
\draw(1,1)-- (1,-0.5);\draw(0,1)-- (0,-0.5);\draw(6.5,1)-- (6.5,-0.5);\draw(8.5,1)-- (8.5,-0.5);\draw(9.5,1)-- (9.5,-0.5);
\filldraw [black] (0.3,0.3) circle (0.8pt) (0.5,0.3) circle (0.8pt) (0.7,0.3) circle (0.8pt);
\filldraw [black] (3.3,0.3) circle (0.8pt) (3.5,0.3) circle (0.8pt) (3.7,0.3) circle (0.8pt);
\filldraw [black] (5.8,0.3) circle (0.8pt) (6,0.3) circle (0.8pt) (6.2,0.3) circle (0.8pt);
\filldraw [black] (8.8,0.3) circle (0.8pt) (9,0.3) circle (0.8pt) (9.2,0.3) circle (0.8pt);
\filldraw [black] (2,2.5) circle (1.4pt)
                  (7.5,2.5) circle (1.4pt)
                  (4.8,-2) circle (1.4pt)
                  (4.8,-3.5) circle (1.4pt);
\draw (-0.55,0.8) node {$v_{4p-2}$}
      (-0.55,-0.3) node {$u_{4p-2}$}
      (0.75,0.8) node {$v_{6}$}
      (0.75,-0.3) node {$u_{6}$}
      (1.25,0.8) node {$v_{2}$}
      (1.25,-0.3) node {$u_{2}$}
      (2.25,0.8) node {$v_{4}$}
      (2.25,-0.3) node {$u_{4}$}
      (2.75,0.8) node {$v_{8}$}
      (2.75,-0.3) node {$u_{8}$}
      (3.7,0.8) node {$v_{4p}$}
      (3.7,-0.3) node {$u_{4p}$}
      (6.0,0.8) node {$v_{4p-1}$}
      (6.0,-0.3) node {$u_{4p-1}$}
      (6.75,0.8) node {$v_{7}$}
      (6.75,-0.3) node {$u_{7}$}
      (7.25,0.8) node {$v_{3}$}
      (7.25,-0.3) node {$u_{3}$}
      (8.25,0.8) node {$v_{1}$}
      (8.25,-0.3) node {$u_{1}$}
      (8.75,0.8) node {$v_{5}$}
      (8.75,-0.3) node {$u_{5}$}
      (10.05,0.8) node {$v_{4p-3}$}
      (10.05,-0.3) node {$u_{4p-3}$};
\draw (2,2.8) node {$u_{4p+1}$}
      (7.5,2.8) node {$u_{4p+2}$}
      (4.8,-2.2) node {$v_{4p+2}$}
      (4.8,-3.7) node {$v_{4p+1}$};
\draw(2,2.5)-- (0,1);\draw(2,2.5)-- (1,1);\draw(2,2.5)-- (1.5,1);\draw(2,2.5)-- (2.5,1);\draw(2,2.5)-- (3,1);\draw(2,2.5)-- (4,1);
\draw(7.5,2.5)-- (5.5,1);\draw(7.5,2.5)-- (6.5,1);\draw(7.5,2.5)-- (7,1);\draw(7.5,2.5)-- (8,1);\draw(7.5,2.5)-- (8.5,1);\draw(7.5,2.5)-- (9.5,1);
\draw(4.8,-2)-- (2.5,-0.5);\draw(4.8,-2)-- (3,-0.5);\draw(4.8,-2)-- (4,-0.5);\draw(4.8,-2)-- (5.5,-0.5);\draw(4.8,-2)-- (6.5,-0.5);\draw(4.8,-2)-- (7,-0.5);
\draw(4.8,-3.5)-- (0,-0.5);\draw(4.8,-3.5)-- (1,-0.5);\draw(4.8,-3.5)-- (1.5,-0.5);\draw(4.8,-3.5)-- (8,-0.5);\draw(4.8,-3.5)-- (8.5,-0.5);\draw(4.8,-3.5)-- (9.5,-0.5);
\draw[-](2,2.5)..controls+(-5.3,-2.2)..(4.8,-3.5);
\draw[-](7.5,2.5)..controls+(-2.7,-1)..(4.8,-2);
\draw[-](2,2.5)..controls+(2.7,-1)..(4.8,-2);
\draw[-](7.5,2.5)..controls+(5.3,-2.2)..(4.8,-3.5);
\end{tikzpicture}
\caption{The subgraph $\widehat{G}_{p+1}$ in the planar decomposition of $K_{4p+2,4p+2}$}
\label{figure 3}
\end{center}
\end{figure}



The remaining case is $n=4p+3$.  A graph $G$ is \textit{thickness $t$-minimal}
if $\theta(G)=t$ and every proper subgraph of $G$ has thickness less than $t$.
When $n=4p+3$, the graph $K_{4p+3,4p+3}$ is a thickness $(p+2)$-minimal graph.
Hobbs and Grossman \cite{HG68} and Bouwer and Broere \cite{BB68} independently
gave two different decompositions of $K_{4p+3,4p+3}$ into planar subgraphs
$\VEC H1{p+2}$.  In each case, each $H_{i}$ for $1\le i\le p+1$ is a
2-connected maximal planar bipartite graph (hence a bar visibility graph), and
the graph $H_{p+2}$ contains only one edge.  Let this edge be $u_{i}v_{j}$ (it
is $u_{1}v_{1}$ in \cite{HG68}  and $u_{4p+3}v_{4p-1}$ in \cite{BB68}).

The bar visibility representation algorithm of \cite{Tamassia86} uses
``$s,t$-numberings'', allowing one to choose any vertex of a bar visibility
graph to be the unique lowest or highest bar in the representation.  Since we
have reduced to the case $n\ge4$, we have $p+1\ge2$.  Choose a representation
of $H_1$ in which $u_i$ is the lowest bar and a representation of $H_2$ in
which $v_j$ is the highest bar.  Place the representation of $H_1$ above the
representation of $H_2$ to incorporate the edge $u_iv_j$ without using an extra
bar for $u_i$ or $v_j$.

We must also show that the bars for $u_i$ in $H_1$ and $v_j$ in $H_2$ can
prevent unwanted visibilities between bars for vertices above and below them.
Since the graph is bipartite, we may assume that bars for the two parts occur
on horizontal lines with those for $U$ having odd vertical coordinates and
those for $V$ having even coordinates.  In addition, the bars on one horizontal
line can extend to meet at endpoints to block visibility between higher and
lower bars for the other part (using both the requirement of positive width for
visibility and the fact that we are representing the complete bipartite graph).
The bars can extend so that on each horizontal line the leftmost occupied point
is the same and the rightmost occupied point is the same.  Now the two
representations can combine as described above.
\end{proof}

\section{General approach to $b(K_{m,n})$}\label{general}

Henceforth let $f(m,n)=\CL{\FR{mn+4}{2m+2n}}$.
Our proof of $b(K_{m,n})\le f(m,n)$ for $m,n\in\NN$ is independent of the
shorter proof for $m=n$ given in the previous section, which relied on
thickness results from earlier papers.  This proof is self-contained.

As mentioned in the introduction, it suffices to produce a $2$-connected
$r$-split of $K_{m,n}$, where $r=f(m,n)$; this is our aim.  We will consider
various cases depending on parity.  In this section we present the common
aspects of the constructions.  We may assume $m\ge n$.  Let the two parts of
$K_{m,n}$ be $X$ and $Y$ with $X=\{\VEC x1m\}$ and $Y=\{\VEC y1n\}$.

When $n$ is even and $m> \frac{1}{2}(n^2-2n-4)$, or when $n$ is odd and
$m>n^2-n-4$, we compute $r=\CL{\FR n2}$.  In this case let $G_i$ be the subgraph
induced by $X\cup\{y_{2i-1},y_{2i}\}$, except that $G_{(n+1)/2}$ is the
subgraph induced by $X\cup\{y_n\}$ when $n$ is odd.  Since $K_{m,2}$ and
$K_{m,1}$ are bar visibility graphs, this decomposes $K_{m,n}$ into
$r$ bar visibility graphs.  Note that $3>3^2-3-4$, so when $n=3$ we have
already considered all cases, and henceforth we may assume $n\ge4$.

We have also considered all cases with $r=\CL{\FR n2}$, so
henceforth we may assume $r\leq \FL{\FR{n-1}2}$.  Let $s = \CL{\FR n2}-r$.
For fixed $n$, the value of $\FR{mn+4}{2m+2n}$ increases with $m$.
Since $m\ge n$, we have $r\ge\CL{\FR{n^2+1}{4n}}\ge\CL{\FR{n+1}4}$.
Thus $s\le r$.  The case $s=r$ requires $s=r=\FR{n+1}4$ and hence
$n\equiv3\mod4$.  For $m\in\{n,n+1,n+2\}$, the values of $\FR{mn+4}{2m+2n}$
are $\FR{n+1}4$, $\FR{n+0.5+15/(4n+2)}4$, and $\FR{n+1+3/(n+1)}4$,
respectively.  The last exceeds $\FR{n+1}4$.  Thus the case $s=r$ occurs
if and only if $n\equiv 3\mod4$ and $m\in\{n,n+1\}$.  Otherwise, $s<n/4<r$.

We will construct a $2$-connected planar graph $G$ that is an $r$-split of
$K_{m,n}$.  In $G$, each vertex will have a label in $X\cup Y$, with each label
used at most $r$ times.  When no vertices labeled $x_i$ and $y_j$ are yet
adjacent, we say that $x_i$ {\it misses} $y_j$; otherwise $x_i$ {\it hits}
$y_j$.  We place vertices in the coordinate plane, with vertices labeled by $X$
on the horizontal axis and vertices labeled by $Y$ on the vertical axis.  Edges
will join only the two axes, so no unwanted edges are formed.  To facilitate
understanding, we first exhibit in Figure~\ref{k87} the graph $G$ that we
produce when $(m,n)=(8,7)$.  For clarity, we record only the subscripts of the
labels on the vertices; the labels are from $X$ on the horizontal axis and from
$Y$ on the vertical axis.

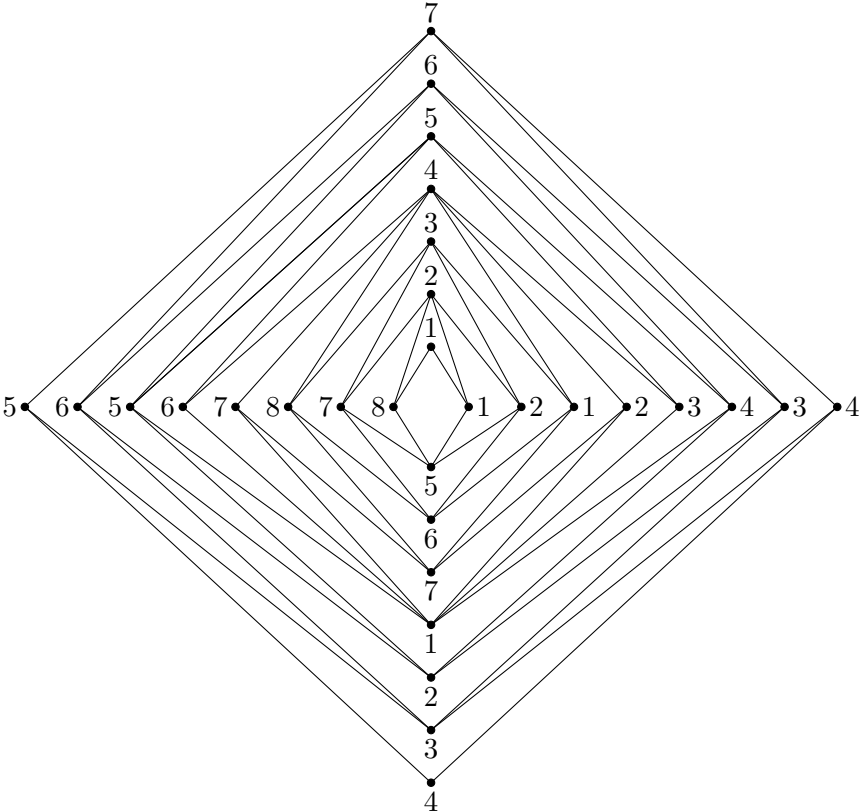
\begin{figure}[htp]
\begin{center}
\begin{tikzpicture}
\begin{scope}[yshift=-9.5cm]

\filldraw [black] (1.7,0) circle (1.4pt)
                   (2.4,0) circle (1.4pt)
                  (3.1,0) circle (1.4pt)
                   (3.8,0) circle (1.4pt)
                   (4.5,0) circle (1.4pt)
                   (5.2,0) circle (1.4pt)
                   (5.9,0) circle (1.4pt)
                  (6.6,0) circle (1.4pt)
                   (7.6,0) circle (1.4pt)
                  (8.3,0) circle (1.4pt)
                   (9,0) circle (1.4pt)
                   (9.7,0) circle (1.4pt)
                   (10.4,0) circle (1.4pt)
                   (11.1,0) circle (1.4pt)
                   (11.8,0) circle (1.4pt)
                   (12.5,0) circle (1.4pt);
\draw (1.5,0) node {\small 5}
 (2.2,0) node {\small 6}
 (2.9,0) node {\small 5}
(3.6,0) node {\small 6}
 (4.3,0) node {\small 7}
 (5,0) node {\small 8}
(5.7,0) node {\small 7}
(6.4,0) node {\small 8}
 (7.8,0) node {\small 1}
 (8.5,0) node {\small 2}
 (9.2,0) node {\small 1}
 (9.9,0) node {\small 2}
 (10.6,0) node {\small 3}
 (11.3,0) node {\small 4}
 (12,0) node {\small 3}
 (12.7,0) node {\small 4};

\filldraw [black] (7.1,0.8) circle (1.4pt)
                   (7.1,1.5) circle (1.4pt)
                   (7.1,2.2) circle (1.4pt)
                   (7.1,2.9) circle (1.4pt)
                   (7.1,3.6) circle (1.4pt)
                   (7.1,4.3) circle (1.4pt)
                   (7.1,5) circle (1.4pt)
                   (7.1,-0.8) circle (1.4pt)
                   (7.1,-1.5) circle (1.4pt)
                   (7.1,-2.2) circle (1.4pt)
                   (7.1,-2.9) circle (1.4pt)
                   (7.1,-3.6) circle (1.4pt)
                   (7.1,-4.3) circle (1.4pt)
                   (7.1,-5) circle (1.4pt);

\draw (7.1,1.05) node {\small 1}
 (7.1,1.75) node {\small 2}
 (7.1,2.45) node {\small 3}
(7.1,3.15) node {\small 4}
 (7.1,3.85) node {\small 5}
 (7.1,4.55) node {\small 6}
(7.1,5.25) node {\small 7}
(7.1,-1.05) node {\small 5}
 (7.1,-1.75) node {\small 6}
 (7.1,-2.45) node {\small 7}
 (7.1,-3.15) node {\small 1}
 (7.1,-3.85) node {\small 2}
 (7.1,-4.55) node {\small 3}
 (7.1,-5.25) node {\small 4};

\draw(7.1,3.6)--(3.8,0);
\draw(1.7,0)-- (7.1,5);\draw(1.7,0)-- (7.1,-4.3);\draw(1.7,0)-- (7.1,-5);
\draw(2.4,0)-- (7.1,5);\draw(2.4,0)-- (7.1,4.3);\draw(2.4,0)-- (7.1,-4.3);\draw(2.4,0)-- (7.1,-3.6);
\draw(3.1,0)-- (7.1,4.3);\draw(3.1,0)-- (7.1,3.6);\draw(3.1,0)-- (7.1,-2.9);\draw(3.1,0)-- (7.1,-3.6);
\draw(3.8,0)-- (7.1,2.9);\draw(3.1,0)-- (7.1,3.6);\draw(3.8,0)-- (7.1,-2.9);
\draw(4.5,0)-- (7.1,2.9);\draw(4.5,0)-- (7.1,-2.2);\draw(4.5,0)-- (7.1,-2.9);
\draw(5.2,0)-- (7.1,2.9);\draw(5.2,0)-- (7.1,2.2);\draw(5.2,0)-- (7.1,-2.2);\draw(5.2,0)-- (7.1,-1.5);
\draw(5.9,0)-- (7.1,1.5);\draw(5.9,0)-- (7.1,2.2);\draw(5.9,0)-- (7.1,-0.8);\draw(5.9,0)-- (7.1,-1.5);
\draw(6.6,0)-- (7.1,1.5);\draw(6.6,0)-- (7.1,0.8);\draw(6.6,0)-- (7.1,-0.8);

\draw(7.6,0)-- (7.1,1.5);\draw(7.6,0)-- (7.1,0.8);\draw(7.6,0)-- (7.1,-0.8);
\draw(8.3,0)-- (7.1,1.5);\draw(8.3,0)-- (7.1,2.2);\draw(8.3,0)-- (7.1,-0.8);\draw(8.3,0)-- (7.1,-1.5);
\draw(9.0,0)-- (7.1,2.9);\draw(9.0,0)-- (7.1,2.2);\draw(9.0,0)-- (7.1,-2.2);\draw(9.0,0)-- (7.1,-1.5);
\draw(9.7,0)-- (7.1,2.9);\draw(9.7,0)-- (7.1,-2.2);\draw(9.7,0)-- (7.1,-2.9);
\draw(10.4,0)-- (7.1,2.9);\draw(10.4,0)-- (7.1,3.6);\draw(10.4,0)-- (7.1,-2.9);
\draw(11.1,0)-- (7.1,4.3);\draw(11.1,0)-- (7.1,3.6);\draw(11.1,0)-- (7.1,-2.9);\draw(11.1,0)-- (7.1,-3.6);
\draw(11.8,0)-- (7.1,5);\draw(11.8,0)-- (7.1,4.3);\draw(11.8,0)-- (7.1,-4.3);\draw(11.8,0)-- (7.1,-3.6);
\draw(12.5,0)-- (7.1,5);\draw(12.5,0)-- (7.1,-4.3);\draw(12.5,0)-- (7.1,-5);
\end{scope}
\end{tikzpicture}
\caption{A bar visibility graph that is a $2$-split of $K_{8,7}$} \label{k87}
\end{center}
\end{figure}

{\bf The plan:}
We first construct subgraphs separately in each half-plane bounded by the
vertical axis.  Combining these two subgraphs along the vertical axis will
yield a $2$-connected plane graph $\wG$ with $rn+sm$ vertices such that
labels in $X$ occur $s$ times and labels in $Y$ occur $r$ times.  Ideally,
each $x_i\in X$ will hit $n-2(r-s)$ different vertices of $Y$, and the vertices
of $Y$ that $x_i$ misses will form $r-s$ pairs such that each pair lies on a
face of length $4$.  We will then insert a copy of $x_i$ in each such face,
adjacent to the missed vertices of $Y$, so that $x_i$ now hits all $n$ vertices
of $Y$.  This brings the usage of each label to $r$ vertices, and the result
will be a $2$-connected $r$-split of $K_{m,n}$.  Because the parity of
$n-2(r-s)$ depends on the parity of $n$, we will need to use different
building blocks for even $n$ and odd $n$.

\medskip
We begin by addressing the matter of $2$-connectedness.

\begin{lemma}\label{2conna}
Let $\wG$ be a plane graph whose vertex set is comprised of sets $A^+$, $A^-$,
$B^+$, and $B^-$ placed along the positive and negative horizontal axes and
the positive and negative vertices axes, respectively (as in Figure~\ref{k87}).
If the four subgraphs induced by $B^+\cup A^+$, $A^+\cup B^-$, $B^-\cup A^-$,
and $A^-\cup B^+$ are connected (and there are no other edges), then $\wG$
is $2$-connected.
\end{lemma}
\begin{proof}
We show that any two vertices $u$ and $v$ in $\wG$ are connected by two
internally disjoint paths.  Consider the four subgraphs combined as in
Figure~\ref{k87}; the graph in each quadrant is connected.  Suppose first that
$u$ and $v$ are not on the same half-axis.  Choose edges $uu'$ and $vv'$ such
that $u'$ and $v'$ are on half-axes different from each other and from $u$ and
$v$.  Now the four vertices are on distinct half-axes, and in every case we
have chosen our two edges from the subgraphs in opposite quadrants.  Because
the remaining two quadrants are connected, we can choose a path in each to
connect the vertices we have chosen on its half-axes.  Now our four chosen
vertices lie on a cycle, which contains the two desired $u,v$-paths.

If $u$ and $v$ lie on the same half-axis, then we can choose $uu'$ and $vv'$
so that $u'$ and $v'$ are on the same neighboring half-axis.  Now the four
vertices are in the boundary of the same quadrant.  We find a $u,v$-path in
one neighboring quadrant and a $u',v'$-path in the other neighboring quadrant.
Again the four vertices lie on a cycle.
\end{proof}

To facilitate computations, we want to reduce to the critical values of $m$.
The next lemma shows that examples for smaller $m$ will cause no difficulty.

\begin{lemma}\label{2conn}
Let $H$ be a $2$-connected plane bipartite graph.  If each part in $H$ has
at least three vertices, and $v\in V(H)$, then edges can be added to
$H-v$ to obtain a $2$-connected plane bipartite graph with the same vertex
bipartition as $H-v$.
\end{lemma}
\begin{proof}
Let $X$ and $Y$ be the parts of the bipartition of $H$; we may assume $v\in X$.
If $H-v$ is $2$-connected, then nothing need be done.  Otherwise, $H-v$ has
a cut-vertex $w$.  Since $\{v,w\}$ is a separating set of $H$, in the
embedding of $H$ these two vertices must lie on the same face.
Since $\C X\ge3$, some component $C$ of $H-\{v,w\}$ contains a vertex of $X$
on its outside face in the embedding; let $x$ be such a vertex.  Each other
component $C'$ contains a neighbor of $v$ (in $Y$), which must lie on the
outside face of $C'$ in the embedding.  Make this vertex in each such component
$C'$ adjacent to $x$.  The resulting graph is planar, $2$-connected, and has
the same bipartition as $H-v$.
\end{proof}

\begin{lemma}\label{mmax}
For $n,r\in\NN$ with $r<n/2$, the largest $m$ such that $f(m,n)=r$ is
$\FL{\FR{2rn-4}{n-2r}}$.
\end{lemma}
\begin{proof}
If $\CL{\FR{mn+4}{2m+2n}}=r$, then $r-1<\FR{mn+4}{2m+2n}\le r$,
equivalent to $\FR{2n(r-1)+4}{n-2(r-1)}<m\le \FR{2nr-4}{n-2r}$.
\end{proof}

\begin{lemma}\label{critical}
If $K_{m,n}$ has a $2$-connected planar $r$-split whenever $r<n/2$ and
$m=\FL{\FR{2rn-4}{n-2r}}$, then $b(K_{m,n})=\CL{\FR{mn+4}{2m+2n}}$ for all
$m,n\in\NN$.
\end{lemma}
\begin{proof}
The given $r$-split $G$ yields the claim for $m=\CL{\FR{mn+4}{2m+2n}}$, where
the parts of $K_{m,n}$ are $X$ and $Y$ with $\C X=m$.  By Lemma~\ref{mmax},
we do not need $r$-split for this $n$ with larger $m$.  For smaller $m$,
Lemma~\ref{2conn} allows us iteratively to delete the copies in $G$ of one
vertex of $X$, restoring $2$-connectedness after each vertex deletion.  We
have an $r$-split of the resulting complete bipartite graph.
Hence $b(K_{m,n})\le r$ for all $m$ where $b(K_{m,n})\le r$ is desired.
\end{proof}

Although the details of the construction differ for even and odd $n$, the main
idea is the same, so we can introduce some common notation.

\begin{definition}\label{axes}
For ease of illustration, we squeeze each half-plane into a strip, drawing its
three axis rays along horizontal lines (see Figure~\ref{k1410}).  The vertices
receiving labels in $Y$ are the first $\CL{rn/2}$ integer points on the
positive vertical axis and the first $\FL{rn/2}$ integer points on the negative
vertical axis, called $B^+$ and $B^-$, respectively.  Starting from the origin,
label $B^+$ in order using $\VEC y1n$, through increasing indices cyclically
modulo $n$.  Similarly label $B^-$, but start with $y_{\CL{n/2}+1}$ and again
continue increasing through indices modulo $n$ (see Figure~\ref{k1410}).  The
last labels on $B^+$ and $B^-$ are $\{y_n,y_{\CL{n/2}}\}$, with $y_n$ ending
$B^+$ if $r$ is even and $B^-$ if $r$ is odd.  Each label $y_i$ is used exactly
$r$ times.  The vertices with labels in $X$ are placed at integer points on the
horizontal axis, with $A^+$ and $A^-$ respectively denoting the sets of
positive and negative points used.  Let $A=A^+\cup A^-$ and $B=B^+\cup B^-$.
\end{definition}

\section{The case of even $n$}

As seen in Figure~\ref{k1410}, most vertices in $A$ will have two consecutive
neighbors in $B^+$ and in $B^-$; vertices in the middle row for the horizontal
axis receive two neighbors above and below.  For now ignore the edges added
there for $x_{13}$ and $x_{14}$.  The main part of the construction consists of
special building blocks that enable most vertices in $X$ to hit $n-2(r-s)$
labels in $Y$ using $s$ vertices on the horizontal axis.  In Figure~\ref{k1410}
these use four vertices in $A$ and five vertices in $B^+$ above and $B^-$ below.
Throughout this section, $n$ is even.

\begin{figure}[htp]
\begin{center}
\begin{tikzpicture}
[xscale=1.25]
\draw (1,0) node {$B^+$};
\filldraw [black] (2,0) circle (1.4pt)
                   (2.6,0) circle (1.4pt)
                  (3.2,0) circle (1.4pt)
                   (3.8,0) circle (1.4pt)
                   (4.4,0) circle (1.4pt)
                   (5,0) circle (1.4pt)
                   (5.6,0) circle (1.4pt)
                   (6.2,0) circle (1.4pt)
                   (6.8,0) circle (1.4pt)
                   (7.4,0) circle (1.4pt)
                  (8,0) circle (1.4pt)
                  (8.6,0) circle (1.4pt)
                    (9.2,0) circle (1.4pt)
                   (9.8,0) circle (1.4pt)
                   (10.4,0) circle (1.4pt);

\draw (2,0.4)   node {\small 1}
      (2.6,0.4) node {\small 2}
      (3.2,0.4) node {\small 3}
      (3.8,0.4) node {\small 4}
      (4.4,0.4) node {\small 5}
      (5,0.4)   node {\small 6}
      (5.6,0.4) node {\small 7}
      (6.2,0.4) node {\small 8}
      (6.8,0.4) node {\small 9}
      (7.4,0.4) node {\small 10}
      (8,0.4)   node {\small 1}
      (8.6,0.4) node {\small 2}
      (9.2,0.4) node {\small 3}
      (9.8,0.4) node {\small 4}
      (10.4,0.4) node {\small 5};

\draw (1,-1) node {$A^+$};
\filldraw [black] (2.3,-1) circle (1.4pt)
                   (2.9,-1) circle (1.4pt)
                  (3.5,-1) circle (1.4pt)
                   (4.1,-1) circle (1.4pt)
                   (4.7,-1) circle (1.4pt)
                   (5.3,-1) circle (1.4pt)
                   (5.9,-1) circle (1.4pt)
                   (6.5,-1) circle (1.4pt)
                   (7.1,-1) circle (1.4pt)
                   (7.7,-1) circle (1.4pt)
                  (8.3,-1) circle (1.4pt)
                  (8.9,-1) circle (1.4pt)
                    (9.8,-1) circle (1.4pt);

\draw (2.1,-1) node {\small 1}
      (2.7,-1) node {\small 2}
      (3.3,-1) node {\small 1}
      (3.9,-1) node {\small 2}
      (4.5,-1) node {\small 3}
      (5.1,-1) node {\small 4}
      (5.7,-1) node {\small 3}
      (6.3,-1) node {\small 4}
      (6.9,-1) node {\small 5}
      (7.5,-1) node {\small 6}
      (8.1,-1) node {\small 5}
      (8.7,-1) node {\small 6}
      (9.5,-1) node {\small 13};

\draw (1,-2) node {$B^-$};
\filldraw [black] (2,-2) circle (1.4pt)
                   (2.6,-2) circle (1.4pt)
                  (3.2,-2) circle (1.4pt)
                   (3.8,-2) circle (1.4pt)
                   (4.4,-2) circle (1.4pt)
                   (5,-2) circle (1.4pt)
                   (5.6,-2) circle (1.4pt)
                   (6.2,-2) circle (1.4pt)
                   (6.8,-2) circle (1.4pt)
                   (7.4,-2) circle (1.4pt)
                  (8,-2) circle (1.4pt)
                  (8.6,-2) circle (1.4pt)
                    (9.2,-2) circle (1.4pt)
                   (9.8,-2) circle (1.4pt)
                   (10.4,-2) circle (1.4pt);

\draw (2,-2.4) node {\small 6}
      (2.6,-2.4) node {\small 7}
      (3.2,-2.4) node {\small 8}
      (3.8,-2.4) node {\small 9}
      (4.4,-2.4) node {\small 10}
      (5,-2.4) node {\small 1}
      (5.6,-2.4) node {\small 2}
      (6.2,-2.4) node {\small 3}
      (6.8,-2.4) node {\small 4}
      (7.4,-2.4) node {\small5}
      (8,-2.4) node {\small 6}
      (8.6,-2.4) node {\small 7}
      (9.2,-2.4) node {\small 8}
      (9.8,-2.4) node {\small 9}
      (10.4,-2.4) node {\small 10};

\draw(2.3,-1)-- (2,0);\draw(2.3,-1)-- (2.6,0);\draw(2.3,-1)-- (2,-2);\draw(2.3,-1)-- (2.6,-2);
\draw(2.9,-1)-- (2.6,0);\draw(2.9,-1)-- (3.2,0);\draw(2.9,-1)-- (2.6,-2);\draw(2.9,-1)-- (3.2,-2);
\draw(3.5,-1)-- (3.2,0);\draw(3.5,-1)-- (3.8,0);\draw(3.5,-1)-- (3.2,-2);\draw(3.5,-1)-- (3.8,-2);
\draw(4.1,-1)-- (3.8,0);\draw(4.1,-1)-- (4.4,0);\draw(4.1,-1)-- (3.8,-2);\draw(4.1,-1)-- (4.4,-2);
\draw(4.7,-1)-- (4.4,0);\draw(4.7,-1)-- (5,0);\draw(4.7,-1)-- (4.4,-2);\draw(4.7,-1)-- (5,-2);
\draw(5.3,-1)-- (5,0);\draw(5.3,-1)-- (5.6,0);\draw(5.3,-1)-- (5,-2);\draw(5.3,-1)-- (5.6,-2);
\draw(5.9,-1)-- (5.6,0);\draw(5.9,-1)-- (6.2,0);\draw(5.9,-1)-- (5.6,-2);\draw(5.9,-1)-- (6.2,-2);
\draw(6.5,-1)-- (6.2,0);\draw(6.5,-1)-- (6.8,0);\draw(6.5,-1)-- (6.2,-2);\draw(6.5,-1)-- (6.8,-2);
\draw(7.1,-1)-- (6.8,0);\draw(7.1,-1)-- (7.4,0);\draw(7.1,-1)-- (6.8,-2);\draw(7.1,-1)-- (7.4,-2);
\draw(7.7,-1)-- (7.4,0);\draw(7.7,-1)-- (8,0);\draw(7.7,-1)-- (7.4,-2);\draw(7.7,-1)-- (8,-2);
\draw(8.3,-1)-- (8,0);\draw(8.3,-1)-- (8.6,0);\draw(8.3,-1)-- (8,-2);\draw(8.3,-1)-- (8.6,-2);
\draw(8.9,-1)-- (8.6,0);\draw(8.9,-1)-- (9.2,0);\draw(8.9,-1)-- (8.6,-2);\draw(8.9,-1)-- (9.2,-2);
\draw(9.8,-1)-- (9.8,0);\draw(9.8,-1)-- (9.2,0);\draw(9.8,-1)-- (9.8,-2);\draw(9.8,-1)-- (9.2,-2);
\draw(9.8,-1)-- (10.4,0);\draw(9.8,-1)-- (10.4,-2);
\end{tikzpicture}
\end{center}

\begin{center}
\begin{tikzpicture}
[xscale=1.25]

\draw (1,0) node {$B^+$};
\filldraw [black] (2,0) circle (1.4pt)
                   (2.6,0) circle (1.4pt)
                  (3.2,0) circle (1.4pt)
                   (3.8,0) circle (1.4pt)
                   (4.4,0) circle (1.4pt)
                   (5,0) circle (1.4pt)
                   (5.6,0) circle (1.4pt)
                   (6.2,0) circle (1.4pt)
                   (6.8,0) circle (1.4pt)
                   (7.4,0) circle (1.4pt)
                  (8,0) circle (1.4pt)
                  (8.6,0) circle (1.4pt)
                    (9.2,0) circle (1.4pt)
                   (9.8,0) circle (1.4pt)
                   (10.4,0) circle (1.4pt);
\draw (2,0.4)   node {\small 1}
      (2.6,0.4) node {\small 2}
      (3.2,0.4) node {\small 3}
      (3.8,0.4) node {\small 4}
      (4.4,0.4) node {\small 5}
      (5,0.4)   node {\small 6}
      (5.6,0.4) node {\small 7}
      (6.2,0.4) node {\small 8}
      (6.8,0.4) node {\small 9}
      (7.4,0.4) node {\small 10}
      (8,0.4)   node {\small 1}
      (8.6,0.4) node {\small 2}
      (9.2,0.4) node {\small 3}
      (9.8,0.4) node {\small 4}
      (10.4,0.4) node {\small 5};

\draw (1,-1) node {$A^-$};
\filldraw [black]  (2.6,-1) circle (1.4pt)
                  (3.5,-1) circle (1.4pt)
                   (4.1,-1) circle (1.4pt)
                   (4.7,-1) circle (1.4pt)
                   (5.3,-1) circle (1.4pt)
                   (5.9,-1) circle (1.4pt)
                   (6.5,-1) circle (1.4pt)
                   (7.1,-1) circle (1.4pt)
                   (7.7,-1) circle (1.4pt)
                  (8.3,-1) circle (1.4pt)
                  (8.9,-1) circle (1.4pt)
                    (9.5,-1) circle (1.4pt)
                   (10.1,-1) circle (1.4pt);
\draw (2.3,-1) node {\small 14}
      (3.3,-1) node {\small 12}
      (3.9,-1) node {\small 11}
      (4.5,-1) node {\small 12}
      (5.1,-1) node {\small 11}
      (5.7,-1) node {\small 10}
      (6.3,-1) node {\small 9}
      (6.9,-1) node {\small 10}
      (7.5,-1) node {\small 9}
      (8.1,-1) node {\small 8}
      (8.7,-1) node {\small 7}
      (9.3,-1) node {\small 8}
      (9.9,-1) node {\small 7};

\draw (1,-2) node {$B^-$};
\filldraw [black] (2,-2) circle (1.4pt)
                   (2.6,-2) circle (1.4pt)
                  (3.2,-2) circle (1.4pt)
                   (3.8,-2) circle (1.4pt)
                   (4.4,-2) circle (1.4pt)
                   (5,-2) circle (1.4pt)
                   (5.6,-2) circle (1.4pt)
                   (6.2,-2) circle (1.4pt)
                   (6.8,-2) circle (1.4pt)
                   (7.4,-2) circle (1.4pt)
                  (8,-2) circle (1.4pt)
                  (8.6,-2) circle (1.4pt)
                    (9.2,-2) circle (1.4pt)
                   (9.8,-2) circle (1.4pt)
                   (10.4,-2) circle (1.4pt);

\draw (2,-2.4)   node {\small 6}
      (2.6,-2.4) node {\small 7}
      (3.2,-2.4) node {\small 8}
      (3.8,-2.4) node {\small 9}
      (4.4,-2.4) node {\small 10}
      (5,-2.4)   node {\small 1}
      (5.6,-2.4) node {\small 2}
      (6.2,-2.4) node {\small 3}
      (6.8,-2.4) node {\small 4}
      (7.4,-2.4) node {\small 5}
      (8,-2.4)   node {\small 6}
      (8.6,-2.4) node {\small 7}
      (9.2,-2.4) node {\small 8}
      (9.8,-2.4) node {\small 9}
      (10.4,-2.4) node {\small 10};

\draw(2.6,-1)-- (2,0);\draw(2.6,-1)-- (2.6,0);\draw(2.6,-1)-- (2,-2);\draw(2.6,-1)-- (2.6,-2);\draw(2.6,-1)-- (3.2,0);\draw(2.6,-1)-- (3.2,-2);
\draw(3.5,-1)-- (3.2,0);\draw(3.5,-1)-- (3.8,0);\draw(3.5,-1)-- (3.2,-2);\draw(3.5,-1)-- (3.8,-2);
\draw(4.1,-1)-- (3.8,0);\draw(4.1,-1)-- (4.4,0);\draw(4.1,-1)-- (3.8,-2);\draw(4.1,-1)-- (4.4,-2);
\draw(4.7,-1)-- (4.4,0);\draw(4.7,-1)-- (5,0);\draw(4.7,-1)-- (4.4,-2);\draw(4.7,-1)-- (5,-2);
\draw(5.3,-1)-- (5,0);\draw(5.3,-1)-- (5.6,0);\draw(5.3,-1)-- (5,-2);\draw(5.3,-1)-- (5.6,-2);
\draw(5.9,-1)-- (5.6,0);\draw(5.9,-1)-- (6.2,0);\draw(5.9,-1)-- (5.6,-2);\draw(5.9,-1)-- (6.2,-2);
\draw(6.5,-1)-- (6.2,0);\draw(6.5,-1)-- (6.8,0);\draw(6.5,-1)-- (6.2,-2);\draw(6.5,-1)-- (6.8,-2);
\draw(7.1,-1)-- (6.8,0);\draw(7.1,-1)-- (7.4,0);\draw(7.1,-1)-- (6.8,-2);\draw(7.1,-1)-- (7.4,-2);
\draw(7.7,-1)-- (7.4,0);\draw(7.7,-1)-- (8,0);\draw(7.7,-1)-- (7.4,-2);\draw(7.7,-1)-- (8,-2);
\draw(8.3,-1)-- (8,0);\draw(8.3,-1)-- (8.6,0);\draw(8.3,-1)-- (8,-2);\draw(8.3,-1)-- (8.6,-2);
\draw(8.9,-1)-- (8.6,0);\draw(8.9,-1)-- (9.2,0);\draw(8.9,-1)-- (8.6,-2);\draw(8.9,-1)-- (9.2,-2);
\draw(9.5,-1)-- (9.8,0);\draw(9.5,-1)-- (9.2,0);\draw(9.5,-1)-- (9.8,-2);\draw(9.5,-1)-- (9.2,-2);
\draw(10.1,-1)-- (9.8,0);\draw(10.1,-1)-- (10.4,0);\draw(10.1,-1)-- (9.8,-2);\draw(10.1,-1)-- (10.4,-2);
\end{tikzpicture}
\caption{Pattern for even $n$, shown for $K_{14,10}$ with $(r,s)=(3,2)$.}
\label{k1410}
\end{center}
\end{figure}
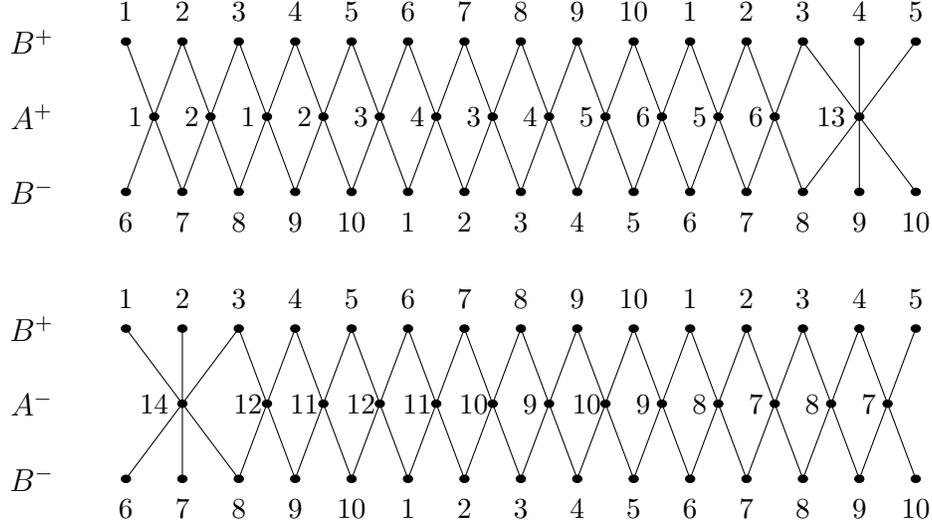

\begin{definition}\label{bricke}
An {\it opposite pair} is a pair of labels in $Y$ whose subscripts differ by
$n/2$; that is, having the form $y_i,y_{i+n/2}$, where the computation in
subscripts is viewed modulo $n$.  The labels of vertices in $B^+$ and $B^-$
at the same distance from the origin form an opposite pair.

An $i$-{\it brick} is a graph induced by $2s$ consecutive vertices in $A^+$
or $A^-$ (with alternating labels $x_{2i-1}$ and $x_{2i}$) and $2s+1$
consecutive vertices in each of $B^+$ and $B^-$ (see Figure~\ref{k1410}).
The vertices used from $B$ form opposite pairs.  The edges of the brick join
the $j$th vertex among its vertices from $A$ to the $j$th and $(j+1)$th
opposite pairs among its vertices from $B$.
\end{definition}

\begin{lemma}\label{bricke1}
When $n$ is even, the labels from $Y$ that lie on a $4$-face in an
$i$-brick form an opposite pair.  Each label for a vertex of $X$ in a brick
hits two intervals of $2s$ cyclically consecutive labels in $Y$, forming $2s$
distinct cyclically consecutive opposite pairs.  The labels missed by such a
vertex of $X$ thus also come in opposite pairs.
\end{lemma}
\begin{proof}
The claims follow immediately from the construction in Definition~\ref{bricke},
because vertices in corresponding positions in the lists $B^+$ and $B^-$ form
opposite pairs.  These pairs are distinct because $2s<n/2$.
\end{proof}


\begin{theorem}\label{even}
If $n$ is even, then $b(K_{m,n})=\CL{\FR{mn+4}{2m+2n}}=r$.
\end{theorem}

\begin{proof}
By Lemma~\ref{critical}, we may assume $m=\FL{\FR{2rn-4}{n-2r}}$.  We have also
reduced to $r<n/2$.  With $s=n/2-r$, we have $s<n/4<r$.  Let
$q=\FR{rn/2-1}{2s}$ and $t=\FL{q}$.  We have $m=\FL{4q}$, so $m=4t+j$ for some
$j\in\{0,1,2,3\}$, where $j$ depends on which fourth of $[0,1)$ contains $q$.

We first put $i$-bricks into $A^+$, for $1\le i\le t$ (from the left in
Figure~\ref{k1410}).  The last vertex of $B^+$ in the $i$-brick is also the
first vertex in the $(i+1)$-brick (similarly for $B^-$).  Thus these bricks use
$1+2st$ vertices from $B^+$ (and $B^-$).  Since $\C{B^+}=rn/2=1+2sq$, there is
room for these bricks.  Similarly, working inward from the outer face (from the
right in Figure~\ref{k1410}), we put $i$-bricks into $A^-$ for $t+1\le i\le 2t$.
Counting the last vertex of the $t$-brick, the number of vertices remaining
visible to unused vertices of $A^+$ at the right end of $B^+$ and $B^-$
is $(rn/2)-2st$, which equals $1+2s(q-t)$.  Similarly, this many vertices are
visible to unused vertices of $A^-$ at the left end.

Since all opposite pairs remain available on the faces with $A^+$ involving the
first $n/2$ vertices in $B^+$ and $B^-$, we have now satisfied the vertices
$\VEC x1{4t}$.  Each such label has been used $s$ times and hit $4s$ labels in
$Y$.  Since $n-4s=2(r-s)$ and the $2(r-s)$ missed labels occur in opposite
pairs, we can add $r-s$ vertices with this label in the appropriate faces to
hit the remaining $2(r-s)$ labels in $Y$.

Since $m=4t+j$, there remain $j$ vertices to process in $X$,
where $j\le 3$ (none if $j=0$; the example in Figure~\ref{k1410} has $j=2$).
Note that the opposite pair in $B^+$ and $B^-$ that is seen from the left
(inner) end of $A^-$ cyclically follows the opposite pair seen from the right
(outer) end of $A^+$.  Thus if a label in $X$ sees consecutive pairs in
$B^+\cup B^-$ using vertices in $A^+$, or in $A^-$, or from the end of $A^+$
and beginning of $A^-$, then the labels in $Y$ hit by that vertex will be
distinct as long as the number of pairs is at most $n/2$.

When $j$ is odd, $x_m$ will receive one vertex at the end of $A^+$ and one
at the beginning of $A^-$.  When $j\ge2$, we assign one vertex in $A^+$ to
$x_{m-j+1}$ and one vertex in $A^-$ to $x_{m-j+2}$.  If this gives
$p$ vertices to a vertex $x_i$ and each of the remaining $r-p$ vertices for
$x_i$ will see one opposite pair by putting it into a $4$-face, then the
specified $p$ vertices for $x_i$ need to hit $n-2(r-p)$ labels in $Y$.  This
value equals $2s+2p$, so $x_i$ needs to hit $s+p$ consecutive opposite pairs.

For $p\le 2$, ensuring that the labels hit are distinct requires $s+2\le n/2$.
Since $s\le (n-2)/4$ when $n$ is even, it suffices to have $(n-2)/4\le n/2-2$,
which is equivalent to $n\ge6$ (and when $n=4$ we cannot have $s<n/4$).

It remains only to show that $B^+\cup B^-$ has enough such pairs available.
For $j\in\{1,2,3\}$, we need in total to hit $s+2$, $2s+2$, or $3s+4$
consecutive pairs, respectively.  We have observed that there are
$1+2s(q-t)$ opposite pairs visible both from the right end of $A^+$ and the
left end of $A^-$.  Since $q-t\ge j/4$, the $2+4s(q-t)$ pairs are at least
$s+2$, $2s+2$, and $3s+2$ for $j\in\{1,2,3\}$, respectively.  However, when
$j=3$ we are using two vertices in each of $A^+$ and $A^-$, meaning that the
last pair seen by one vertex can also be the first pair seen by the other.
This means that in total the vertices can see $3s+4$ pairs instead of $3s+2$,
which is the number needed.


Finally, we must ensure that the graph $\wG$ produced before adding the
excess labels in faces is $2$-connected (it is an elementary exercise that
adding vertices of degree $2$ to a $2$-connected graph preserves
$2$-connectedness).  By Lemma~\ref{2conna}, it suffices to show that the four
subgraphs induced by $B^+\cup A^+$, $A^+\cup B^-$, $B^-\cup A^-$, and
$A^-\cup B^+$ are connected.  After adding the vertices with labels
$\VEC x{m-j+1}m$ in the last step, we may have left some vertices of $B^+$ or
$B^-$ unhit.  Add edges joining $A$ and $B$ in each of the four induced
subgraphs to make them connected, while remaining planar and bipartite with
the same bipartition.  Because we are seeking a representation of a complete
bipartite graph, extra visibilities between the parts do not cause a problem.
\end{proof}

\section{The case of odd $n$}

Our general aim was to have $s$ vertices for $x_i$ hit $n-2(r-s)$ labels in
$Y$.  When $n$ is odd, this amount is odd, so we change the definition of
bricks.  They will still use $2s$ vertices from $A$, but now they will use
one less vertex each in $B^+$ and $B^-$.  Indeed, the bricks we used before
are too big to fit onto $B^+$ and $B^-$.  Throughout this section, $n$ is odd.

\begin{definition}\label{bricko}
A {\it skew pair} is a pair of labels in $Y$ whose subscripts differ by
$\FL{n/2}$; that is, having the form $y_i,y_{i+(n-1)/2}$, where the computation
in subscripts is viewed modulo $n$.

For odd $n$, an $i$-{\it brick} is a graph induced by $2s$ consecutive vertices
in $A^+$ or $A^-$ (alternating labels $x_{2i-1}$ and $x_{2i}$) and $2s$
consecutive vertices in $B^+$ and $B^-$ (see Figure~\ref{k169}).
Bricks using $A^+$ start from the left end (near the origin).  Those using
$A^-$ start from the right (not near the origin).  Vertices from $B^+$ and
$B^-$ in a brick have the same distances from the start.

Let $\{\wB,\vB\}=\{B^+,B^-\}$.
Measured from the start, the edges of a brick join the $j$th copy of $x_{2i-1}$
to the $(2j-1)$th and $2j$th vertices of the brick in $\wB$ and the
$(2j-2)$th and $(2j-1)$th vertices of the brick in $\vB$, except that the first
copy of $x_{2i-1}$ hits in $\vB$ only the first vertex.  Similarly, the $j$th
copy of $x_{2i}$ hits the $2j$th and $(2j+1)$th vertices of the brick in $\wB$
and the $(2j-1)$ and $2j$th vertices of the brick in $\vB$, except that the
$s$th copy of $x_{2i}$ hits in $\wB$ only the last vertex.  For bricks from
the left (using $A^+$), set $\wB=B^+$ and $\vB=B^-$.  For bricks from the
right (using $A^-$), let $\wB$ be the member of $\{B^+,B^-\}$ whose last label
is $y_n$, and let $\vB$ be the member whose last label is $y_{(n+1)/2}$.
\end{definition}

\begin{figure}[htp]
\begin{center}
\begin{tikzpicture}
[xscale=1.2]

\filldraw [black]  (2.0,0) circle (1.4pt)
                   (2.6,0) circle (1.4pt)
                   (3.2,0) circle (1.4pt)
                   (3.8,0) circle (1.4pt)
                   (4.4,0) circle (1.4pt)
                   (5.0,0) circle (1.4pt)
                   (5.6,0) circle (1.4pt)
                   (6.2,0) circle (1.4pt)
                   (6.8,0) circle (1.4pt)
                   (7.4,0) circle (1.4pt)
                   (8.0,0) circle (1.4pt)
                   (8.6,0) circle (1.4pt)
                   (9.2,0) circle (1.4pt)
                   (9.8,0) circle (1.4pt);

\draw (1,0) node {$B^+$};
\draw (2,0.4)   node {\small 1}
      (2.6,0.4) node {\small 2}
      (3.2,0.4) node {\small 3}
      (3.8,0.4) node {\small 4}
      (4.4,0.4) node {\small 5}
      (5,0.4)   node {\small 6}
      (5.6,0.4) node {\small 7}
      (6.2,0.4) node {\small 8}
      (6.8,0.4) node {\small 9}
      (7.4,0.4) node {\small 1}
      (8,0.4)   node {\small 2}
      (8.6,0.4) node {\small 3}
      (9.2,0.4) node {\small 4}
      (9.8,0.4) node {\small 5};

\filldraw [black]  (2.3,-1) circle (1.4pt)
                   (2.7,-1) circle (1.4pt)
                   (3.1,-1) circle (1.4pt)
                   (3.5,-1) circle (1.4pt)
                   (4.1,-1) circle (1.4pt)
                   (4.5,-1) circle (1.4pt)
                   (4.9,-1) circle (1.4pt)
                   (5.3,-1) circle (1.4pt)
                   (5.9,-1) circle (1.4pt)
                   (6.3,-1) circle (1.4pt)
                   (6.7,-1) circle (1.4pt)
                   (7.1,-1) circle (1.4pt)
                   (7.7,-1) circle (1.4pt)
                   (8.1,-1) circle (1.4pt)
                   (8.5,-1) circle (1.4pt)
                   (8.9,-1) circle (1.4pt);

\draw (1,-1) node {$A^+$};
\draw              (2.1,-1) node {\small $1$}
                   (2.5,-1) node {\small $2$}
                   (2.9,-1) node {\small $1$}
                   (3.3,-1) node {\small $2$}
                   (3.9,-1) node {\small $3$}
                   (4.3,-1) node {\small $4$}
                   (4.7,-1) node {\small $3$}
                   (5.1,-1) node {\small $4$}
                   (5.7,-1) node {\small $5$}
                   (6.1,-1) node {\small $6$}
                   (6.5,-1) node {\small $5$}
                   (6.9,-1) node {\small $6$}
                   (7.5,-1) node {\small $7$}
                   (7.9,-1) node {\small $8$}
                   (8.3,-1) node {\small $7$}
                   (8.7,-1) node {\small $8$};

\filldraw [black]  (2.0,-2) circle (1.4pt)
                   (2.6,-2) circle (1.4pt)
                   (3.2,-2) circle (1.4pt)
                   (3.8,-2) circle (1.4pt)
                   (4.4,-2) circle (1.4pt)
                   (5.0,-2) circle (1.4pt)
                   (5.6,-2) circle (1.4pt)
                   (6.2,-2) circle (1.4pt)
                   (6.8,-2) circle (1.4pt)
                   (7.4,-2) circle (1.4pt)
                   (8.0,-2) circle (1.4pt)
                   (8.6,-2) circle (1.4pt)
                   (9.2,-2) circle (1.4pt);

\draw (1,-2) node {$B^-$};
\draw (2,-2.4)   node {\small 6}
      (2.6,-2.4) node {\small 7}
      (3.2,-2.4) node {\small 8}
      (3.8,-2.4) node {\small 9}
      (4.4,-2.4) node {\small 1}
      (5,-2.4)   node {\small 2}
      (5.6,-2.4) node {\small 3}
      (6.2,-2.4) node {\small 4}
      (6.8,-2.4) node {\small 5}
      (7.4,-2.4) node {\small 6}
      (8,-2.4)   node {\small 7}
      (8.6,-2.4) node {\small 8}
      (9.2,-2.4) node {\small 9};

\draw(2.3,-1)-- (2,0);\draw(2.3,-1)-- (2.6,0);\draw(2.3,-1)-- (2,-2);
\draw(2.7,-1)-- (2.6,0);\draw(2.7,-1)-- (3.2,0);
\draw(2.7,-1)-- (2.6,-2);\draw(2.7,-1)-- (2,-2);
\draw(3.1,-1)-- (3.2,0);\draw(3.1,-1)-- (3.8,0);
\draw(3.1,-1)-- (3.2,-2);\draw(3.1,-1)-- (2.6,-2);
\draw(3.5,-1)-- (3.8,0); \draw(3.5,-1)-- (3.8,-2);\draw(3.5,-1)-- (3.2,-2);
\draw(4.1,-1)-- (3.8,0);\draw(4.1,-1)-- (3.8,-2);\draw(4.1,-1)-- (4.4,0);
\draw(4.5,-1)-- (4.4,0);\draw(4.5,-1)-- (5.0,0);
\draw(4.5,-1)-- (4.4,-2);\draw(4.5,-1)-- (3.8,-2);
\draw(4.9,-1)-- (5.6,0);\draw(4.9,-1)-- (5.0,0);
\draw(4.9,-1)-- (5.0,-2);\draw(4.9,-1)-- (4.4,-2);
\draw(5.3,-1)-- (5.6,0);\draw(5.3,-1)-- (5.0,-2);\draw(5.3,-1)-- (5.6,-2);

\draw(5.9,-1)-- (5.6,0);\draw(5.9,-1)-- (6.2,0);\draw(5.9,-1)-- (5.6,-2);
\draw(6.3,-1)-- (6.2,0);\draw(6.3,-1)-- (6.8,0);
\draw(6.3,-1)-- (6.2,-2);\draw(6.3,-1)-- (5.6,-2);
\draw(6.7,-1)-- (6.8,0);\draw(6.7,-1)-- (7.4,0);
\draw(6.7,-1)-- (6.8,-2);\draw(6.7,-1)-- (6.2,-2);
\draw(7.1,-1)-- (7.4,0); \draw(7.1,-1)-- (7.4,-2);\draw(7.1,-1)-- (6.8,-2);
\draw(7.7,-1)-- (7.4,0);\draw(7.7,-1)-- (7.4,-2);\draw(7.7,-1)-- (8.0,0);
\draw(8.1,-1)-- (8.0,0);\draw(8.1,-1)-- (8.6,0);
\draw(8.1,-1)-- (8.0,-2);\draw(8.1,-1)-- (7.4,-2);
\draw(8.5,-1)-- (9.2,0);\draw(8.5,-1)-- (8.6,0);
\draw(8.5,-1)-- (8.6,-2);\draw(8.5,-1)-- (8.0,-2);
\draw(8.9,-1)-- (9.2,0);\draw(8.9,-1)-- (8.6,-2);\draw(8.9,-1)-- (9.2,-2);

\draw[dashed](8.9,-1)-- (9.8,0);
\end{tikzpicture}
\end{center}

\begin{center}
\begin{tikzpicture}
[xscale=1.2]

\draw (1,0) node {$B^+$};\draw (1,-1) node {$A^-$};\draw (1,-2) node {$B^-$};

\filldraw [black] (2.0,0) circle (1.4pt)
                   (2.6,0) circle (1.4pt)
                   (3.2,0) circle (1.4pt)
                   (3.8,0) circle (1.4pt)
                   (4.4,0) circle (1.4pt)
                   (5.0,0) circle (1.4pt)
                   (5.6,0) circle (1.4pt)
                   (6.2,0) circle (1.4pt)
                   (6.8,0) circle (1.4pt)
                   (7.4,0) circle (1.4pt)
                   (8.0,0) circle (1.4pt)
                   (8.6,0) circle (1.4pt)
                   (9.2,0) circle (1.4pt)
                   (9.8,0) circle (1.4pt);

\draw (2,0.4)   node {\small 1}
      (2.6,0.4) node {\small 2}
      (3.2,0.4) node {\small 3}
      (3.8,0.4) node {\small 4}
      (4.4,0.4) node {\small 5}
      (5,0.4)   node {\small 6}
      (5.6,0.4) node {\small 7}
      (6.2,0.4) node {\small 8}
      (6.8,0.4) node {\small 9}
      (7.4,0.4) node {\small 1}
      (8,0.4)   node {\small 2}
      (8.6,0.4) node {\small 3}
      (9.2,0.4) node {\small 4}
      (9.8,0.4) node {\small 5};

\filldraw [black]  (2.9,-1) circle (1.4pt)
                   (3.3,-1) circle (1.4pt)
                   (3.7,-1) circle (1.4pt)
                   (4.1,-1) circle (1.4pt)
                   (4.7,-1) circle (1.4pt)
                   (5.1,-1) circle (1.4pt)
                   (5.5,-1) circle (1.4pt)
                   (5.9,-1) circle (1.4pt)
                   (6.5,-1) circle (1.4pt)
                   (6.9,-1) circle (1.4pt)
                   (7.3,-1) circle (1.4pt)
                   (7.7,-1) circle (1.4pt)
                   (8.3,-1) circle (1.4pt)
                   (8.7,-1) circle (1.4pt)
                   (9.1,-1) circle (1.4pt)
                   (9.5,-1) circle (1.4pt);

\draw              (2.6,-1) node {\small $16$}
                   (3.1,-1) node {\small $15$}
                   (3.5,-1) node {\small $16$}
                   (3.9,-1) node {\small $15$}
                   (4.5,-1) node {\small $14$}
                   (4.9,-1) node {\small $13$}
                   (5.3,-1) node {\small $14$}
                   (5.7,-1) node {\small $13$}
                   (6.3,-1) node {\small $12$}
                   (6.7,-1) node {\small $11$}
                   (7.1,-1) node {\small $12$}
                   (7.5,-1) node {\small $11$}
                   (8.1,-1) node {\small $10$}
                   (8.5,-1) node {\small $9$}
                   (8.9,-1) node {\small $10$}
                   (9.3,-1) node {\small $9$};

\filldraw [black]  (2.6,-2) circle (1.4pt)
                   (3.2,-2) circle (1.4pt)
                   (3.8,-2) circle (1.4pt)
                   (4.4,-2) circle (1.4pt)
                   (5.0,-2) circle (1.4pt)
                   (5.6,-2) circle (1.4pt)
                   (6.2,-2) circle (1.4pt)
                   (6.8,-2) circle (1.4pt)
                   (7.4,-2) circle (1.4pt)
                   (8.0,-2) circle (1.4pt)
                   (8.6,-2) circle (1.4pt)
                   (9.2,-2) circle (1.4pt)
                   (9.8,-2) circle (1.4pt);

\draw (2.6,-2.4) node {\small 6}
      (3.2,-2.4) node {\small 7}
      (3.8,-2.4) node {\small 8}
      (4.4,-2.4) node {\small 9}
      (5,-2.4)   node {\small 1}
      (5.6,-2.4) node {\small 2}
      (6.2,-2.4) node {\small 3}
      (6.8,-2.4) node {\small 4}
      (7.4,-2.4) node {\small 5}
      (8,-2.4)   node {\small 6}
      (8.6,-2.4) node {\small 7}
      (9.2,-2.4) node {\small 8}
      (9.8,-2.4) node {\small 9};

\draw(2.9,-1)-- (2.6,-2); \draw(2.9,-1)-- (3.2,-2);  \draw(2.9,-1)-- (2.6,0);
\draw(3.3,-1)-- (3.2,-2); \draw(3.3,-1)-- (3.8,-2);
\draw(3.3,-1)-- (3.2,0);  \draw(3.3,-1)-- (2.6,0);
\draw(3.7,-1)-- (3.8,-2); \draw(3.7,-1)-- (4.4,-2);
\draw(3.7,-1)-- (3.8,0);  \draw(3.7,-1)-- (3.2,0);
\draw(4.1,-1)-- (4.4,-2); \draw(4.1,-1)-- (4.4,0);  \draw(4.1,-1)-- (3.8,0);

\draw(4.7,-1)-- (4.4,-2); \draw(4.7,-1)-- (4.4,0);  \draw(4.7,-1)-- (5.0,-2);
\draw(5.1,-1)-- (5.0,-2); \draw(5.1,-1)-- (5.6,-2);
\draw(5.1,-1)-- (5,0);    \draw(5.1,-1)-- (4.4,0);
\draw(5.5,-1)-- (6.2,-2); \draw(5.5,-1)-- (5.6,-2);
\draw(5.5,-1)-- (5.6,0);  \draw(5.5,-1)-- (5,0);
\draw(5.9,-1)-- (6.2,-2); \draw(5.9,-1)-- (5.6,0);  \draw(5.9,-1)-- (6.2,0);

\draw(6.5,-1)-- (6.2,-2); \draw(6.5,-1)-- (6.8,-2);  \draw(6.5,-1)-- (6.2,0);
\draw(6.9,-1)-- (6.8,-2); \draw(6.9,-1)-- (7.4,-2);
\draw(6.9,-1)-- (6.8,0);  \draw(6.9,-1)-- (6.2,0);
\draw(7.3,-1)-- (7.4,-2); \draw(7.3,-1)-- (8.0,-2);
\draw(7.3,-1)-- (7.4,0);  \draw(7.3,-1)-- (6.8,0);
\draw(7.7,-1)-- (8.0,-2); \draw(7.7,-1)-- (8,0);\draw(7.7,-1)-- (7.4,0);

\draw(8.3,-1)-- (8.0,-2); \draw(8.3,-1)-- (8,0);\draw(8.3,-1)-- (8.6,-2);
\draw(8.7,-1)-- (8.6,-2); \draw(8.7,-1)-- (9.2,-2);
\draw(8.7,-1)-- (8.6,0);  \draw(8.7,-1)-- (8,0);
\draw(9.1,-1)-- (9.8,-2); \draw(9.1,-1)-- (9.2,-2);
\draw(9.1,-1)-- (9.2,0);  \draw(9.1,-1)-- (8.6,0);
\draw(9.5,-1)-- (9.8,-2); \draw(9.5,-1)-- (9.2,0);\draw(9.5,-1)-- (9.8,0);

\draw[dashed](2.9,-1)-- (2,0);

\end{tikzpicture}
\caption{Pattern for odd $n$, shown for $K_{16,9}$ with $(r,s)=(3,2)$.}
\label{k169}
\end{center}
\end{figure}
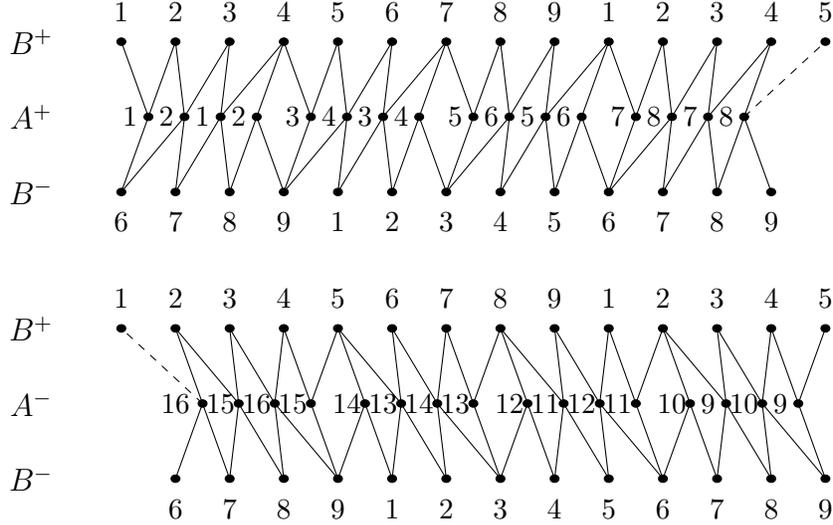

\begin{lemma}\label{bricko1}
When $n$ is odd, the labels from $Y$ that lie on a $4$-face in an $i$-brick
form a skew pair.  Each label for a vertex of $X$ in a brick hits two intervals
of $2s-2$ cyclically consecutive labels in $Y$, forming $2s-2$ cyclically
consecutive skew pairs, plus one more label at the end of one of those
intervals.  These labels are distinct.  The labels missed by such a vertex of
$X$ also come in skew pairs.
\end{lemma}
\begin{proof}
The labels on a $4$-face in a brick are the $j$th vertex of $\wB$ and $(j+1)$th
of $\vB$ from the start.  Since $B^+$ starts with $y_1$ and $B^-$ starts with
$y_{(n+3)/2}$, for bricks using $A^+$ the two labels on the face are $y_{j+1}$
and $y_{j+(n+1)/2}$, which form a skew pair taking subscripts modulo $n$.
For bricks using $A^-$, starting from the other end, $\wB$ starts with $y_n$
and $\vB$ starts with $y_{(n+1)/2}$.  (As specified in Definition~\ref{axes},
$y_n$ ends $B^+$ if $r$ is even and $B^-$ if $r$ is odd.)
On a $4$-face in such a brick we have $y_{n-j}$ and $y_{(n+1)/2-j-1}$.
Since $(n-1)/2-j\equiv n-j+(n-1)/2\mod{n}$, again we have a skew pair.

Since corresponding positions in $B^+$ and $B^-$ are labeled by skew pairs,
the $4s$ labels in $Y$ occurring in a brick are distinct unless $2s=(n+1)/2$,
which can occur when $n\equiv3\mod 4$ and $s=r=(n+1)/4$.  In this case,
the first label from $B^+$ is the same as the last label from $B^-$ in a
brick.  However, as constructed in Definition~\ref{bricko}, the first label
from $B^+$ is hit only by $x_{2i-1}$, and the last label from $B^-$ is hit only
by $x_{2i}$ in the brick, so each label from $X$ still hits $4s-1$ distinct
labels in $Y$.

These $4s-1$ distinct labels group into $2s-1$ cyclically consecutive skew
pairs plus one more label.  The two intervals of labels hit by the pairs leave
two intervals of labels missed, and the lengths of the intervals of missed
labels are $(n+1)/2-s+1$ and $(n-1)/2-s+1$.  The extra label hit by $x_i$ is at
the end of the longer interval.  No matter which end of the longer interval
it shortens, the remaining missed labels match up as skew pairs.
\end{proof}

The approach to the construction is the same as for even $n$ in
Theorem~\ref{even}, but the technical details are different.

\begin{theorem}\label{odd}
If $n$ is odd, then $b(K_{m,n})=\CL{\FR{mn+4}{2m+2n}}=r$.
\end{theorem}

\begin{proof}
By Lemma~\ref{critical}, we may assume $m=\FL{\FR{2rn-4}{n-2r}}$.
We have reduced to $\FR{n+1}4\le r<\FR n2$.  With $s=\FR{n+1}2-r$, we have
$s\le\FR{n+1}4\le r$.  Also $n-2r=2s-1$.  Let $q=\FR{rn/2-1}{2s-1}$ and
$t=\FL{q}$.  We have $m=\FL{4q}$, so $m=4t+j$ for some $j\in\{0,1,2,3\}$,
depending on where in $[0,1)$ is $q$.

We put $i$-bricks into $A^+$ for $1\le i\le t$ (from the left in
Figure~\ref{k169}).  The last vertex of $B^+$ in the $i$-brick is the
first vertex in the $(i+1)$-brick (similarly for $B^-$).  Thus these bricks use
$1+(2s-1)t$ vertices from $B^+$ (and $B^-$).  Since $B^+$ and $B^-$ each have
at least $(rn-1)/2$ vertices, and $(rn-1)/2=(1/2)+(2s-1)q\ge (1/2)+(2s-1)t$,
there is room for these bricks.  Similarly, working inward from the outer face
(from the right in Figure~\ref{k169}), we put $i$-bricks into $A^-$ for
$t+1\le i\le 2t$.  Counting the last vertex of the $t$-brick, the number of
vertices remaining visible to unused vertices of $A^+$ at the right end of
$B^+$ and $B^-$ together is $rn-2(2s-1)t$, which equals $2+2(2s-1)(q-t)$.
Similarly, this many vertices are visible to unused vertices of $A^-$ at the
left end.

Since all skew pairs remain available on the faces with $A^+$ involving the
first $n$ vertices in $B^+$ and $B^-$, we have now satisfied the vertices
$\VEC x1{4t}$.  Each such label has been used $s$ times and hit $4s-1$ labels
in $Y$.  Since $n+1-4s=2(r-s)$ and the $2(r-s)$ missed labels occur in skew
pairs, we can add $r-s$ vertices with this label in the appropriate faces to
hit the remaining $2(r-s)$ labels in $Y$.

Since $m=4t+j$, there remain $j$ vertices to process in $X$, where $j\le 3$
(none if $j=0$; as in the example in Figures~\ref{k87} and~\ref{k169}).
The labels that end $B^+$ and $B^-$ and may be visible at the end of $A^+$
are $\{y_n,y_{(n+1)/2}\}$, a skew pair.  The labels that begin $B^+$ and $B^-$
are $\{y_1,y_{(n+3)/2}\}$, the next skew pair.  Thus if a label in $X$ sees
consecutive skew pairs in $B^+\cup B^-$ using vertices in $A^+$, or in $A^-$,
or from the end of $A^+$ and beginning of $A^-$, then the labels in $Y$ hit by
that vertex will be distinct as long as the number of pairs is at most $n/2$.

When $j$ is odd, $x_m$ will receive one vertex at the end of $A^+$ and one
at the beginning of $A^-$.  When $j\ge2$, we assign one vertex in $A^+$ to
$x_{m-j+1}$ and one vertex in $A^-$ to $x_{m-j+2}$.  If this assigns $p$
vertices to a vertex $x_i$ and each of the remaining $r-p$ vertices for $x_i$
will see one skew pair by putting it into a $4$-face, then the specified $p$
vertices for $x_i$ need to hit $n-2(r-p)$ labels in $Y$.  This value equals
$2s-1+2p$, so it suffices for $x_i$ to hit $s+p-1$ consecutive skew pairs and
one label from the next pair.

For $p\le 2$, ensuring that the labels hit are distinct requires
$s+1\le(n-1)/2$.  Since $s\le (n+1)/4$, it suffices to have $(n+1)/4\le(n-3)/2$,
which is equivalent to $n\ge7$.  Since we have reduced to $n\ge4$, and
$s\le(n-1)/4$ when $n\equiv1\mod{4}$, all cases are covered.

It remains only to show that $B^+\cup B^-$ has enough vertices available.
For $j\in\{1,2,3\}$, we need in total to hit $2s+3$, $4s+2$, or $6s+5$ labels,
respectively.  We have observed that there are in total $2+2(2s-1)(q-t)$
vertices of $B$ visible both from the right end of $A^+$ and the left end of
$A^-$.  Since $q-t\ge j/4$, the total number of vertices is
at least $4+(2s-1)j$, which is enough when $j\le2$.  When $j=3$ we are using
two vertices in each of $A^+$ and $A^-$, meaning that the last pair seen by one
vertex can also be the first pair seen by the other.  This provides four
additional visibilities to reach the needed $6s+5$.

Finally, we must ensure that the graph $\wG$ produced before adding the
excess labels in faces is $2$-connected.  Here the argument applying
Lemma~\ref{2conna} to the subgraphs induced by $B^+\cup A^+$, $A^+\cup B^-$,
$B^-\cup A^-$, and $A^-\cup B^+$ is the same as in Theorem~\ref{even}.
\end{proof}

\bibliography{bibfile}
\end{document}